\documentclass[a4paper]{amsart}

\usepackage{dsfont}
\usepackage{amsmath, amsthm, amssymb}
\usepackage{xypic}
\usepackage{graphicx}
\usepackage{color}
\usepackage{mathtools}
\usepackage[usenames,dvipsnames]{xcolor}
\usepackage[latin1]{inputenc}
\usepackage{indentfirst}
\usepackage{epstopdf}
\usepackage{subfigure}
\usepackage{emptypage}
\usepackage{hyperref}
\hypersetup{
    colorlinks=true, 
    linktoc=all,     
    linkcolor=blue,  
    citecolor=JungleGreen,
}

\newtheorem{theorem}{Theorem}

\newtheorem*{lemmacontactinflating}{Lemma \ref{lem:contactinflating}}
\newtheorem*{propositionfolding}{Proposition \ref{prop:folding}}
\newtheorem{corollary}[theorem]{Corollary}

\newtheorem{proposition}[theorem]{Proposition}
\newtheorem{lemma}[theorem]{Lemma}
\newtheorem{definition}[theorem]{Definition}
\newtheorem{remark}[theorem]{Remark}

\newtheorem{theorem*}{Theorem}
\newtheorem{question*}[theorem*]{Question}
\newtheorem{conjecture*}[theorem*]{Conjecture}
\newtheorem{corollary*}[theorem*]{Corollary}
\newtheorem{theorem*e}{Teorema}
\newtheorem{question*e}[theorem*e]{Pregunta}
\newtheorem{conjecture*e}[theorem*e]{Conjetura}
\newtheorem{corollary*e}[theorem*e]{Corolario}

\renewcommand{\ker}{\mathrm{ker}}

\newcommand{\interior}{\mathrm{int}}

\newcommand{\id}{\mathrm{id}}

\newcommand{\supp}{\mathrm{supp}}

\newcommand{\R}{\mathds{R}}
\newcommand{\Z}{\mathds{Z}}
\newcommand{\N}{\mathds{N}}

\newcommand{\Q}{\mathds{Q}}
\newcommand{\T}{\mathds{T}}
\newcommand{\diff}{\mathrm{Diff}}
\newcommand{\symp}{\mathrm{Symp}}
\newcommand{\cont}{\mathrm{Cont}}
\newcommand{\ham}{\mathrm{Ham}}
\newcommand{\G}{\mathcal{G}}
\newcommand{\A}{\mathcal{A}}
\newcommand{\C}{\mathcal{C}}
\newcommand{\M}{\mathcal{M}}
\newcommand{\E}{\mathcal{E}}
\newcommand{\x}{\textbf{x}}
\newcommand{\y}{\textbf{y}}
\newcommand{\0}{\textbf{0}}

\begin{document}

\title{The conjugation method in symplectic dynamics}

\author{Luis Hern\'{a}ndez--Corbato}
\address{Instituto de Ciencias Matematicas CSIC--UAM--UCM--UC3M, C. Nicol\'{a}s Cabrera, 13--15, 28049, Madrid, Spain}
\email{luishcorbato@mat.ucm.es}

\author{Francisco Presas}
\address{Instituto de Ciencias Matematicas CSIC--UAM--UCM--UC3M, C. Nicol\'{a}s Cabrera, 13--15, 28049, Madrid, Spain}
\email{fpresas@icmat.es}



\begin{abstract}
We prove the existence of minimal symplectomorphisms and strictly ergodic contactomorphisms on manifolds
which admit a locally free $\mathbb{S}^1$--action by symplectomorphisms and contactomorphisms, respectively.
The proof adapts
the conjugation method, introduced by Anosov and Katok, to the contact
and symplectic setting.
\end{abstract}

\maketitle

\section{Introduction}

Symplectic dynamics has been an intense research area for the last fifty or sixty years.
It all started with Poincar\'{e}--Birkhoff Theorem in the first decades of last century.
Arnold conjecture and related questions were pivotal to the revolution in this field, though.
The existence of extra fixed points for Hamiltonian symplectomorphisms, statement of Arnold conjecture, has
been proved in most of the instances.
The Hamiltonian flows also tend to have special dynamical properties. For instance, the existence of periodic orbits
on a dense set of energy levels (see Hofer and Zehnder \cite{hoferzehnder}).
There is another instance of Hamiltonian rigidity, that is the Conley conjecture (proven in several cases,
see \cite{ginzburggurel}). The conjecture states that a Hamiltonian symplectomorphism has an infinite number of
geometrically different periodic orbits. The mantra behind this note is to point out that the removal of the
Hamiltonian condition destroys most of the rigidity in dynamics. The prototypical example is $\T^2$ with the symplectic
(non Hamiltonian) irrational translation: it is minimal and, therefore, it does not have a single periodic orbit. We will generalize this example.

Rigidity in contact dynamics is more subtle.
Hamiltonian contact flows do correspond to a special kind of Hamiltonian symplectic flows.
Therefore, the existence result of periodic orbits in a generic energy level was conjecturally strengthen to
all the levels in the convex case. The generalization of this fact is the content of the famous Weinstein conjecture. It has been a powerful engine developing
contact topology in the last 30 years. It has been proven in several cases.
However, the behaviour of a discrete contact dynamical system remained vague. Only, recently, an analogue of the Arnold conjecture
was stated: the translated points conjecture (see \cite{sandon}).
In this case, our aim is to show that the interplay between the Reeb flow and the Hamiltonian contactomorphisms,
which is the content of the conjecture, is probably the only natural rigid phenomenum in discrete Hamiltonian contact dynamics.
In particular, we will show that (positive) Hamiltonian contactomorphisms can be pretty wild from a dynamical point of view.

The approximation by conjugation method, introduced by Anosov and Katok, is a remarkable technique to
produce diffeomorphisms with prescribed properties on the asymptotical distribution of their orbits.
In their seminal paper \cite{anosovkatok}, Anosov and Katok proved the existence of ergodic transitive diffeomorphisms
on any closed manifold. The name of the method comes from the fact that the maps are constructed as limits
of conjugates $h R_{\alpha} h^{-1}$ of a non--trivial $\mathbb{S}^1$--action $\{R_{\alpha}: \alpha \in \mathbb{S}^1\}$ on the manifold.
This direct approach was replaced by an abstract one by Fathi and Herman in \cite{fathiherman}. They applied
the Baire category theorem to the closure of the previous set of conjugates to prove the existence of minimal and uniquely ergodic
diffeomorphisms in closed manifolds that admit a locally free $\mathbb{S}^1$--action. A short overview of the conjugation method
together with further applications is found in \cite{fayadkatok}.

In this article, the arguments in \cite{fathiherman} are adapted to the symplectic and contact setting.
Denote by $\symp(M, \omega)$ the group of symplectomorphisms of a symplectic manifold $(M, \omega)$
and $\symp_0(M, \omega)$ the connected component of the identity map. Similarly, denote by $\cont(M, \xi)$
the group of contactomorphisms of a contact manifold $(M, \xi)$ and $\cont_0(M, \xi)$ the connected component
of the identity. All these groups are known to be $C^1$--closed, hence $C^{\infty}$--closed, in $\diff(M)$.

As it will be clear from the contents of the article a more ambitious goal would be to adapt the statements
to the Hamiltonian symplectomorphism group\footnote{As is well--known, the Hamiltonian contactomorphism group
corresponds to the identity component of the contactomorphisms group, i.e. any contact vector field is Hamiltonian.}. This encounters two obstacles:
\begin{itemize}
\item $\ham(M, \omega)$ is not known to be $C^{\infty}$--closed in $\diff(M)$ as this is the content of the Flux
conjecture. However, this can be dealt with in particular cases (for instance, assuming that $M$ is simply connected).
\item A Hamiltonian $\mathbb{S}^1$--action is never locally free. Therefore, the conjugation method, particularly simple in \cite{fathiherman}, needs to be sophisticated \cite{fayadkatok}.
This requires more subtle symplectic topology results to be developed in a forthcoming work.
\end{itemize}

A map is called minimal if every orbit is dense and is called uniquely ergodic provided it has
just one invariant probability measure.
The main theorems of this article are:

\begin{theorem}\label{thm:symplectic}
If a symplectic manifold $(M, \omega)$ admits a locally free $\mathbb{S}^1$--action by symplectomorphisms, then
there exists $\psi \in \symp_0(M)$ such that $\psi$ is minimal.
\end{theorem}

\begin{theorem}\label{thm:contact}
If a contact manifold $(M, \xi)$ admits a locally free $\mathbb{S}^1$--action by contactomorphisms, then
there exists $\psi \in \cont_0(M, \xi)$ such that $\psi$ is strictly ergodic.
\end{theorem}

Given a contact manifold $(M, \xi)$, $\xi = \ker(\alpha)$ cooriented, a path/loop of contactomorphisms
$\{\psi_t\}$ is \emph{positive} if $\alpha (\partial \psi/\partial t) > 0$ everywhere.
By definition, this is equivalent to the Hamiltonian $H$ which generates this path/loop being everywhere positive.
The $\mathbb{S}^1$--action on $M$ is said positive if the loop of contactomorphisms it defines is positive.

\begin{corollary}
If the $\mathbb{S}^1$--action is positive then $\psi$ from Theorem \ref{thm:contact} is generated by a positive path.
\end{corollary}
\begin{proof}
This follows from a general fact for contact manifolds that admit a positive loop.
The group of contactomorphisms is locally path connected (see \cite[Chapter 6]{banyaga}) hence $\cont_0(M, \xi)$ is path connected.
Let $\{\varphi_t\}_{t = 0}^1$ be a path of contactomorphisms from $\varphi_0 = \id$ to $\varphi_1 = \psi$.
Denote by $\{R_{t}\}_{t=0}^1$ the loop of contactomorphisms given by the action.
A computation (see \cite{casalspresas})
 shows that, if $H, G: M \times [0,1] \to \R$ are the Hamiltonians which generate $\{\varphi_t\}$ and $\{R_t\}$,
the composition path $\{\varphi_t \circ R_{mt}\}_{t = 0}^1$ is generated by the Hamiltonian
$$F(p, t) = H(p, t) + m e^{-f_t} G(\varphi_t^{-1}(p), t),$$
where $\varphi_t^* \alpha = e^{f_t} \alpha$ and $m \ge 1$ is arbitrary. Incidentally, notice that $G$ is independent of $t$.
Then, since $G$ is strictly positive, if $m \ge 1$ is sufficiently large, $F > 0$.
\end{proof}

It will be obvious from the version of the conjugation method used in this note that:
\begin{corollary}\label{cor:id}
The diffeomorphisms constructed in Theorems \ref{thm:symplectic} and \ref{thm:contact} can be chosen $C^{\infty}$--close
to the identity.
\end{corollary}

\textbf{Examples}

{\bf 1. Prequantum contact manifolds.}

By a result of Thomas \cite{thomas}, a contact manifold admits a locally free $\mathbb{S}^1$--action if and only if
it is the prequantization $\mathbb{S}^1$--bundle over a symplectic orbifold (also called circle orbibundle).

{\bf 2. Manifolds not admitting a positive $\mathbb{S}^1$--action.}

The spherical cotangent bundle, $\mathbb{S} T^* M$, of a manifold $M$ (also referred to as oriented projectivization of $T^* M$)
carries a canonical cooriented contact structure. There do not exist positive paths of Legendrian isotopies connecting
different fibers of $\mathbb{S} T^* M$ provided the universal cover of $M$ is $\R^n$ (see \cite{colin}) or, more generally,
an open manifold (see \cite{chernov}).
As a consequence, $\mathbb{S} T^* M$ does not admit positive loops. 
However, it may admit $\mathbb{S}^1$--actions: any $\mathbb{S}^1$--action on $M$ by diffeomorphisms lifts to a $\mathbb{S}^1$--action on $\mathbb{S} T^* M$ by contactomorphisms so Theorem \ref{thm:contact} can be applied.
This is the case, for instance, of standard $\T^3$ as it is the spherical cotangent bundle over $\T^2$.

{\bf 3. Symplectic examples.}

If $M$ admits a symplectic free $\mathbb{S}^1$--action,
then it admits a fibration $H\colon M \to \mathbb{S}^1$ such that the action preserves its fibers. Moreover,
there is a $C^{\infty}$-small perturbation $\widetilde{\omega}$ of the original symplectic form for which the action is still symplectic
and such that $\ker (\iota_{X} \widetilde{\omega}) = \ker (dH)$,
where $X$ denotes the infinitesimal generator of the action.
Consequently, there exists a symplectic fibration $\widehat{H}\colon M/\mathbb{S}^1 \to \mathbb{S}^1$.
Partial converse results are available \cite{marisafdez}.


\section*{Acknowledgements}
The authors express their gratitude to Viktor Ginzburg and Daniel Peralta for their useful comments during
the preparation of the manuscript.
The authors have been supported by the Spanish Research Projects SEV-2015-0554 and MTM2013-42135 and by the
ERC Starting Grant 335079 from Daniel Peralta.

\section{Conjugation method}\label{sec:conjugationmethod}

The conjugation method was introduced by Anosov and Katok \cite{anosovkatok} in 1970.
They constructed ergodic diffeomorphisms on every closed manifold as limits of volume preserving periodic transformations.
Later, Fathi and Herman \cite{fathiherman} developed the method to prove the existence of
minimal and uniquely ergodic volume preserving diffeomorphisms in manifolds which admit a locally free $\mathbb{S}^1$--action.
Katok had previously announced \cite{katokrussian} the theorem provided the action is free.
This section, following the approach of \cite{fathiherman}, presents how the conjugation method is used
to find minimal and uniquely ergodic diffeomorphisms.

Given a closed manifold $N$, $\diff(N)$ denotes the group of diffeomorphisms of $N$ equipped with the $C^{\infty}$--topology.
As our aim is to adapt the method to work for subgroups of $\diff(N)$, an abstract subgroup $\G(N) < \diff(N)$ is considered.
These subgroups must be defined, at least, for the manifold in consideration, later denoted $M$, and some quotients
of $M$ over the free action of a finite group.
In \cite{fathiherman}, $\G(N)$ was set to be equal to $\diff(N)$ or $\diff_{\mu}(N)$
(diffeomorphisms preserving a prescribed measure $\mu$ of positive $C^{\infty}$ density).
In this work, $\G(N)$ will later be set to be $\symp_0(M, \omega)$ or $\cont_0(M, \xi)$ and the underlying geometric structure
will be preserved by the finite quotients under consideration.
Although the previously cited subgroups are closed, this discussion does not assume $\G(N)$ is necessarily closed,
$\overline{\G(N)}$ denotes the closure of $\G(N)$.
The topology considered in the group of diffeomorphisms and their subsets will always be the $C^{\infty}$--topology.

Let $M$ be a closed manifold and, for simplicity, put $\G = \G(M)$.
Assume $M$ admits a locally free $\mathbb{S}^1$--action on $M$ by diffeomorphisms in $\G$, i.e.
a group homomorphism $\mathbb{S}^1 \rightarrow \G: \alpha \mapsto R_{\alpha}$ such that for every $x \in M$ the stabilizer subgroup
$S_x = \{\alpha: R_{\alpha}(x) = x\}$ is a discrete subgroup of $\mathbb{S}^1$ (in particular, it is finite).
Throughout this article, $\mathbb{S}^1$ is identified to $\R/\Z$.

The union of all stabilizers $S = \cup_{x \in M} S_x$ is still a finite subgroup of $\mathbb{S}^1$.
Indeed, for any $x \in M$ consider a small neighborhood $V_x$ of $S_x$ in $\mathbb{S}^1$ such that any subgroup of $\mathbb{S}^1$ contained
in $V_x$ is a subgroup of $S_x$.
By continuity one gets that $S_y \subset V_x$ in some neighborhood of $x$, so $S_y < S_x$.
The conclusion follows from the compactness of $M$.

The goal is to prove the existence of elements of $\overline{\G}$ satisfying certain dynamical properties. Let $\A$ be the subset of
$\overline{\G}$ composed of such elements. The conjugation method aims to prove that $\A$ is not empty by showing that
the elements of $\A$ are generic in the subgroup
\begin{equation}\label{eq:Cdefinition}
\C = \overline{\{g R_{\alpha} g^{-1}: \alpha \in \mathbb{S}^1, g \in \G\}}
\end{equation}
of $\overline{\G}$, that is, $\A \cap \C$ is a dense $G_{\delta}$ subset of $\C$
(intersection of countably many open and dense subsets).
Corollary \ref{cor:id} is then a trivial consequence of this approach.
Recall that since $\diff(M)$ is a Baire space and $\C$ is closed in $\overline{\G}$ and so in $\diff(M)$, $\C$ is also a Baire space.
Thus, it is enough to find a countable family $\{A_j\}_{j \in J}$ of open and dense subsets of $\C$
such that
$$\bigcap_{j \in J} A_j = \A \cap \C.$$

In the proofs below, a family of open and dense subsets $\{A_i\}_{i \in I}$ of $\C$ is defined.
It has a countable coinitial subfamily $\{A_j\}_{j \in J}$, that is, for every $i \in I$ there is $j \in J$
such that $A_j \subset A_i$.
Consequently, $\cap_{i \in I} A_i = \cap_{j \in J} A_j$.
Additionally, for any $g \in \G$ and $i \in I$, $g A_i g^{-1} \subset A_{k}$ for some $k \in I$.
In order to prove the density of $A_i$ in $\C$, one needs to show that every element $g R_{\alpha} g^{-1}$
belongs to the closure of $A_i$. In view of the previous properties on the family $\{A_i\}_{i \in I}$,
it is enough to show that $R_{\alpha} \in \overline{A_i}$ for any $\alpha \in \mathbb{S}^1$ and $i \in I$.

A number $\alpha \in \R/\Z$ is said to be \emph{coprime} with $S$ if $S \cap \langle \alpha \rangle = \{0\}$,
where $\langle \alpha \rangle$ denotes the subgroup of $\mathbb{S}^1 \cong \R/\Z$ generated by $\alpha$.
Since the set $S$ of stabilizers is finite and $\Q/\Z$ is dense in $\mathbb{S}^1$, it suffices to check that
$R_{\alpha} \in \overline{A_i}$ for $\alpha \in \Q/\Z$ coprime with $S$ in order to prove it for all $\alpha \in \R/\Z$.

\subsection{Minimal diffeomorphisms}

This subsection shows how the conjugation method was applied in \cite{fathiherman} to obtain diffeomorphisms
with all their orbits dense.

\begin{definition}
A homeomorphism $f \colon X \to X$ is said to be minimal if every orbit is dense, i.e.
$M = \overline{\{f^n(x): n \in \Z\}}$ for every $x \in X$.
\end{definition}
Minimality is a topological property which is easily seen to be equivalent to the non-existence of proper
closed subsets invariant under the dynamics.
\begin{lemma}\label{lem:minimalequivalence}
Let $X$ be a compact topological space.
A homeomorphism $f\colon X \to X$ is minimal if and only if for every open set $U \subset X$ there exists $n \in \N$
such that
$$U \cup f(U) \cup \ldots \cup f^n(U) = X.$$
\end{lemma}
\begin{proof}[Sketch of the proof.]
Notice that the complement of $\cup_{m \in \Z} f^m(U)$ in $X$ is closed and $f$--invariant, it must be empty if $f$ is minimal.
The open cover $\{f^m(U)\}_{m \in \Z}$ of $X$ has, by compactness, a finite subcover whose elements, after applying $f$ conveniently,
are forward iterates of $U$.
\end{proof}

Let us apply the conjugation method to obtain minimal diffeomorphisms in $\overline{\G}$.
Denote by $\M$ the set of minimal diffeomorphisms. Consider the sets
$$W_{U, n} = \{g \in \G: U \cup g(U) \cup \ldots \cup g^n(U) = M\},$$
where $U$ ranges over the open subsets of $M$ and $n$ over the positive integers.
By definition, $W_{U,n}$ is an open subset of $\diff(M)$.
Consequently, $\M_U = \overline{\G} \, \cap \, \left(\cup_{n \ge 1} W_{U, n}\right)$ is an open subset of $\overline{\G}$.
By Lemma \ref{lem:minimalequivalence}, $\M$ is the intersection of all open sets $\M_U$.
Since $M$ has a countable basis of open sets $\{U_i\}_{i\in \N}$, the subfamily $\{\M_{U_i}\}_{i \in \N}$ is coinitial
and $\M = \cap_i \M_{U_i}$. Moreover, $g \M_U g^{-1} = \M_{g^{-1}(U)}$.
Thus, in order to check that $\M_U$ is dense in $\C$ it is enough to prove that $R_{\alpha}$ is accumulated by elements
of $\M_U$ for every $\alpha \in (\Q/\Z) \setminus S$.

Given any $\alpha = p/q \notin S$, $\mathrm{gcd}(p,q) = 1$, $F_q$ denotes the subgroup of $\mathbb{S}^1$ generated by $p/q$.
The quotient of $M$ under the action of $F_q$ is a manifold $\widehat{M} = M/F_q$. The $\mathbb{S}^1$--action on $M$ induces,
by the identification $\mathbb{S}^1/F_q \cong \mathbb{S}^1$, another locally free $\mathbb{S}^1$--action on $\widehat{M}$.
This action is given by $\{\widehat{R}_{\beta}: \beta \in \mathbb{S}^1\}$, where $\widehat{R}_{q\alpha}$
is the projection of $R_{\alpha}$ onto $\widehat{M}$.
The subgroup $\widehat{\G} := \G(\widehat{M}) < \diff(\widehat{M})$ must satisfy the following hypothesis:

\medskip
\textbf{(H1)}
Any element of $\widehat{\G}$ has a lift in $\G = \G(M)$.
\medskip

Next condition must be also fulfilled. These hypothesis are discussed in Section \ref{sec:proofs}.
Note that it is an exercise to check them in the case $\G(N) = \diff(N)$ or $\G(N) = \diff_{\mu}(N)$,
where $\mu$ is a probability measure on $N$ of positive $C^{\infty}$ density.

\medskip
\textbf{(H2)} For any locally free $\mathbb{S}^1$--action on a manifold $N$, and any open $V \subset N$, there exists an
element $f \in \G(N)$ such that $f^{-1}(V)$ meets all the orbits of the action.
\medskip

Fix an open set $U \subset M$. For the previous $\alpha = p/q$, denote by $\widehat{U}$ the projection of $U$ onto $\widehat{M}$
and let $\widehat{h}$ be the map supplied by (H2) for $N = \widehat{M}$,
$V = \widehat{U}$ and the induced $\mathbb{S}^1$--action on $\widehat{M}$.
There is a lift $h \in \G$ of $\widehat{h}$ whose existence is guaranteed by (H1). It satisfies
$h R_{\alpha} h^{-1} = R_{\alpha}$ and $h^{-1}(U)$ meets all the orbits of the $\mathbb{S}^1$--action on $M$ because $\alpha$
is coprime with $S$.

The following lemma generalizes to $\mathbb{S}^1$--actions a simple fact:
the iterates of an open interval under an irrational rotation eventually cover $\mathbb{S}^1$.

\begin{lemma}\label{lem:irrationalminimal}
For any $\beta \notin \Q/\Z$, $h R_{\beta} h^{-1} \in \M_U$.
\end{lemma}
\begin{proof}
Let $\gamma$ denote a orbit of the action and $W$ an open subset of $M$, $\gamma \cap W \neq \emptyset$.
There exists $n \ge 1$ such that $W \cup R_{\beta}(W) \cup \ldots \cup R_{n\beta}(W)$ cover $\gamma$.
The same is true for orbits close to $\gamma$.
Take $W = h^{-1}(U)$, then $W$ meets every orbit.
Since the orbit space of the action is compact, there exists $m \ge 1$ such that
$W \cup R_{\beta}(W) \cup \ldots \cup R_{\beta}^m(W)$ contains every orbit $\gamma$.
Thus $h R_{\beta} h^{-1} \in W_{U, m} \cap \C \subset \M_U$.
\end{proof}

Take a sequence of irrational numbers $\beta_n \to \alpha$.
Clearly, $h R_{\beta_n} h^{-1} \xrightarrow{\C^{\infty}} h R_{\alpha} h^{-1} = R_{\alpha}$.
Since from Lemma \ref{lem:irrationalminimal}, $h R_{\beta_n} h^{-1}$ belongs to $\M_U$ and also, by definition, to $\C$,
the map $R_{\alpha} \in \overline{\M_U}$ and the conclusion follows.

In conclusion, the existence of a minimal diffeomorphism in $\overline{\G}$ is guaranteed provided (H1) and (H2) are satisfied.

\subsection{Strictly ergodic diffeomorphisms}

In this subsection, the conjugation method is applied to obtain a diffeomorphism
whose orbits are uniformly distributed along $M$ in a measure theory sense.

\begin{definition}
Let $X$ be a compact metric space. A homeomorphism $f\colon X \to X$ is called uniquely ergodic if it has only one
invariant probability measure. If the invariant measure has full support $f$ is called strictly ergodic.
\end{definition}

Next lemma follows from two dynamical facts: the support of an $f$--invariant measure is itself invariant under $f$
and any compact invariant subset of $X$ admits an invariant measure supported on it.

\begin{lemma}
A map is strictly ergodic if and only if it is uniquely ergodic and minimal.
\end{lemma}

Once strict ergodicity has been split into two properties, let us take care of unique ergodicity.

\begin{proposition}\label{prop:uecharacterization}
Let $X$ be a compact metric space and $f\colon X \to X$ be a homeomorphism. The following statements are equivalent.
\begin{itemize}
\item $f$ is uniquely ergodic.
\item For every map $u \in C^0(X, \R)$,
$\frac{1}{n} \left(\sum_{k = 0}^{n-1} u \circ f^k\right)$ converges uniformly
as $n \to \infty$ to a constant (equal to $\int_X u \, d\mu$,
where $\mu$ is the only invariant probability measure).
\end{itemize}
\end{proposition}
\begin{proof}
See Walters \cite[Theorem 6.19]{walters}
\end{proof}

Since the space of real--valued continuous functions over $M$ is separable, Proposition \ref{prop:uecharacterization}
can be used to show that the set of uniquely ergodic diffeomorphisms $\E$ is a $G_{\delta}$ subset of $\diff(M)$.
Indeed, let $u \in C^0(X, \R)$, $\varepsilon > 0$. Define
$$\E(u, \varepsilon) = \biggl\{f \in \C: \exists \, n \ge 1, c \in \R \text{ such that }
\biggl|\biggl| \frac{1}{n} \sum_{k = 0}^{n-1} u \circ f^k - c\biggr|\biggr| < \varepsilon \biggr\},$$
where $\C$, the closure of the set of conjugates of the action, was defined in (\ref{eq:Cdefinition}).
Trivially, the sets $\E(u, \varepsilon)$ are open.
Fix a dense sequence $\{u_i\}$ in $C^0(X, \R)$.
\begin{lemma}
$$\E \cap \C = \bigcap_i \bigcap_{k \ge 1} \E(u_i, 1/k).$$
\end{lemma}
\begin{proof}
Simply notice that if $\varepsilon_1 > \varepsilon_0 > 0$ and $\Vert u_1 - u_0 \Vert < \varepsilon_1 - \varepsilon_0$
then $\E(u_0, \varepsilon_0) \subset \E(u_1, \varepsilon_1)$.
\end{proof}
Furthermore, note that $g \, \E(u, \varepsilon) \, g^{-1} = \E(u \circ g, \varepsilon)$.
Thus, in order to show that $\E$ is dense in $\C$ it suffices to check that $R_{\alpha}$ belongs to
$\overline{\E(u, \varepsilon)}$ for every rational $\alpha$ not contained in $S$.

Henceforth, assume $u \in C^0(M, \R)$ and $\varepsilon > 0$ are fixed and the following hypothesis is satisfied.

\medskip
\textbf{(H3)}
Given a locally free $\mathbb{S}^1$--action $\{R_{\alpha}\}_{\alpha \in \mathbb{S}^1}$ on a manifold $N$, an open set $V \subset N$
and $\varepsilon > 0$, there exists an element $f \in \G(N)$ such that
$$\lambda(\{\alpha \in \mathbb{S}^1: R_{\alpha}(y) \notin f^{-1}(V)\}) < \varepsilon$$
for every $y \in N$ ($\lambda$ denotes the Lebesgue measure in $\mathbb{S}^1$ of total mass equal to 1).
\medskip

There are some unavoidable bothering technical issues to check this condition
(see \cite[Section 6]{fathiherman}).

As in the minimal case, fix $\alpha = p/q \notin S$, $\gcd(p, q) = 1$, and write $F_q = \langle \alpha \rangle$. Define
$$\widehat{u} = \frac{1}{q} \sum_{k = 0}^{q-1} u \circ R_{k/q}.$$
Clearly, $\widehat{u}$ can be seen as a function on the quotient $\widehat{M} = M/F_q$.
Consider an arbitrary $\eta > 0$ and fix $\delta > 0$ such that $\delta (1 + ||u||) < \eta$.
Take $y_0 \in \widehat{M}$ and note that
$\widehat{U} = \{y \in \widehat{M}: |\widehat{u}(y) - \widehat{u}(y_0)| < \delta\}$
is a non--empty open subset of $\widehat{M}$.
Apply (H3) to find $\widehat{h} \in \widehat{\G}$ such that for every $y \in \widehat{M}$,
$\lambda(\{\alpha \in \mathbb{S}^1: \widehat{R}_{\alpha}(y) \notin \widehat{h}^{-1}(U)\}) < \delta$,
where $\widehat{R}$ denotes the $\mathbb{S}^1$--action on $\widehat{M}$ induced by $R$. Then,
\[
\left|\int_{\mathbb{S}^1} \widehat{u} \circ \widehat{h} \circ \widehat{R}_{\beta}(y) d\beta - \widehat{u}(y_0)\right| \le
\delta + \delta ||\widehat{u}|| \le \delta (1 + ||u||) < \eta.
\]
for every $y \in \widehat{M}$.
Choose a lift $h \in \G = \G(M)$ of $\widehat{h}$, whose existence is guaranteed by (H1). Then,
$h \circ R_{\alpha} = R_{\alpha} \circ h$, and, for every $x \in M$,
$$\int_{\mathbb{S}^1} u \circ h \circ R_{\theta}(x) d\theta =
\int_{\mathbb{S}^1} \widehat{u} \circ \widehat{h} \circ \widehat{R}_{\beta}(\widehat{x}) d\beta,$$
where $\widehat{x}$ denotes the projection of $x$ onto $\widehat{M}$.
In sum, we have proved the following lemma.

\begin{lemma}\label{lem:averagingmap}
For any $\eta > 0$, there exists $h \in \G$ and a constant $c \in \R$ such that
$$\left| \int_{\mathbb{S}^1} u \circ h \circ R_{\theta}(x) d\theta - c\right| < \eta$$
holds for every $x \in M$.
\end{lemma}

\begin{proposition}
$R_{\alpha} \in \overline{\E(u, \varepsilon)}$.
\end{proposition}
\begin{proof}
Since an irrational rotation is uniquely ergodic, for any $\beta \notin \Q/\Z$ and $x \in M$
$$\lim_{n \to \infty} \frac{1}{n} \sum_{k = 0}^{n-1} v \circ R_{\beta}^k(x) = \int_{\mathbb{S}^1} v \circ R_{\theta} d\theta$$
holds for every $v \in C^0(M, \R)$.
Furthermore, the convergence is uniform in $x \in M$.
In particular, if $h$ comes from Lemma \ref{lem:averagingmap}
\[
\lim_{n \to \infty} \frac{1}{n} \sum_{k = 0}^{n-1} u \circ h \circ R_{\beta}^k \left( h^{-1} (x) \right) =
\int_{\mathbb{S}^1} (u \circ h) \circ R_{\theta}(h^{-1}(x)) d\theta.
\]
and the convergence is uniform.
If $\eta$ from Lemma \ref{lem:averagingmap} is chosen sufficiently small and $n$ is large
\begin{align*}
\left\Vert \frac{1}{n} \sum_{k = 0}^{n-1} u \circ \left( h R_{\beta}^k h^{-1} \right) - c \right\Vert &
\le \left\Vert \frac{1}{n} \sum_{k = 0}^{n-1} u \circ h \circ R_{\beta}^k h^{-1} -
\int_{\mathbb{S}^1} u \circ h \circ R_{\theta} d\theta \right\Vert + \\
 & + \left\Vert \int_{\mathbb{S}^1} u \circ h \circ  R_{\theta} d\theta - c \right\Vert < \varepsilon,
\end{align*}
so $h R_{\beta} h^{-1} \in \E(u, \varepsilon)$.
The conclusion is obtained taking a sequence of irrational $\beta_n \to \alpha$ because
$h R_{\beta_n} h^{-1} \xrightarrow{\C^{\infty}} h R_{\alpha} h^{-1} = R_{\alpha}$.
\end{proof}

In conclusion, $\overline{\G}$ contains strictly ergodic diffeomorphisms as long as it satisfies (H1) and (H3).

\section{Proofs of main theorems}\label{sec:proofs}

The strategy of the proofs of Theorems \ref{thm:symplectic} and \ref{thm:contact} is the
conjugation method which was explained in Section \ref{sec:conjugationmethod}.
There have been some marked points in the arguments, where hypothesis on the diffeomorphisms subgroups $\G$
were assumed and must be checked separately.

%

Theorem \ref{thm:symplectic} follows once it is shown that (H1) and (H2) are valid for the connected component of the
group of symplectic diffeomorphisms which contains the identity.
Analogously, to prove Theorem \ref{thm:contact} (H1) and (H3) must be shown to hold true for
the connected component of the identity in the group of contactomorphisms.
Notice, incidentally, that (H3) implies (H2).

\begin{definition}
The action of a group $G < \diff(M)$ is said to be $n$--transitive if for every $x_1, \ldots, x_n$ and
$y_1, \ldots, y_n$ there is $g \in G$ such that $g(x_i) = y_i$, $i = 1\ldots n$.
\end{definition}
The following result is well-known, see Boothby \cite{boothby} for a detailed proof.
\begin{theorem}\label{thm:transitivity}
The group of symplectic/contact diffeomorphisms acts $n$--transitively for any $n \ge 1$.
Furthermore, the same is true for the group of Hamiltonian symplectomorphisms and Hamiltonian contactomorphisms.
\end{theorem}

\medskip
\textbf{Hypothesis (H1):}
\emph{Any element of $\widehat{\G} = \G(\widehat{M})$ has a lift in $\G = \G(M)$.}
\medskip

This is obviously true for contactomorphisms and symplectomorphisms.
Recall that $\widehat{M}$ was defined as the quotient of $M$ under the free action of the group generated
by $R_{p/q}$.



\bigskip
\textbf{Hypothesis (H2):}
\emph{Given any locally free $\mathbb{S}^1$--action on a manifold $N$, and any open $V \subset N$, there exists an
element $f \in \G(N)$ such that $f^{-1}(V)$ meets all the orbits of the action.}
\medskip

It will be now checked that this hypothesis is verified in the case
$(N^{2n}, \omega)$ is a symplectic manifold, the action is symplectic and $\G(N) = \symp_0(N, \omega)$.

A Darboux flow box for a symplectic vector field $X$ is a Darboux chart $(U_i, \theta_i)$,
$$\theta_i \colon U_i \to Q(\delta, \rho) := [-\delta, \delta] \times [-\rho, \rho] \times B^2(\rho) \times \ldots B^2(\rho) \subset \R^{2n}$$ such that $\theta^*_i \bigl(\frac{\partial}{\partial x_1} \bigr) = X$.
Henceforth, $X$ will be the infinitesimal generator of the $\mathbb{S}^1$--action.
Darboux flow boxes do exist at any point of $M$.
The next statement follows from compactness of $M$.

\begin{lemma}\label{lem:flowboxes}
There exist $\varepsilon, r > 0$ and $\{(U_i, \vartheta_i)\}_{i = 1}^m$
a finite set of pairwise disjoint Darboux flow boxes such that
$\vartheta_i: U_i \to Q(\varepsilon, r + \varepsilon)$ and the union of the codimension--1 disks
$$\bigsqcup_{i=1}^m \vartheta_i^{-1}(\{0\}\times [-r,r] \times B^2(r) \times \ldots \times B^2(r))$$
touches all the orbits of $X$.
\end{lemma}

This lemma allows to work in a local fashion in $(\R^{2n}, \omega_0)$ so as to apply the following
squeezing result, which will be proved in Section \ref{sec:packing}.

\begin{proposition}\label{prop:folding}
Let $r > 0$, $D = \{0\} \times [-r, r] \times B^2(r) \times \ldots \times B^2(r) \subset \R^{2n}$
and $\varepsilon > 0$. For any $\delta > 0$, there exists a Hamiltonian symplectomorphism $\psi$
with support in
$Q(\varepsilon, r) = [-\varepsilon, \varepsilon] \times [-r-\varepsilon, r + \varepsilon] \times B^2(r + \varepsilon) \times
\ldots B^2(r + \varepsilon)$ such that $\psi(D) \subset P^{2n}(\delta, \ldots, \delta)$.
\end{proposition}

Apply Lemma \ref{lem:flowboxes} to obtain $r, \varepsilon > 0$ and a family of Darboux flow boxes $\{(U_i, \vartheta_i)\}_{i =1}^m$.
By Theorem \ref{thm:transitivity}, there is a Hamiltonian symplectomorphism $\phi$ such that $\phi^{-1}(V)$ contains the centers
$\vartheta_i^{-1}(\0)$ of the flow boxes.
Let $\psi$ be the squeezing map given by Proposition \ref{prop:folding}. Define a Hamiltonian
symplectomorphism $\varphi$ in $N$ which is equal to $\vartheta_i \circ \psi \circ \vartheta_i^{-1}$ in
$U_i$, for any $1 \le i \le m$, and to the identity elsewhere.
Then $(\phi \circ \varphi)^{-1}(V)$ meets all the orbits of the $\mathbb{S}^1$--action.

\begin{remark}
Note that the map $\phi \circ \varphi$ in $\symp_0(N)$ realizing (H2) is actually a Hamiltonian symplectomorphism.
Ultimately, this boils down to the fact that both $n$--transitivity and Proposition \ref{prop:folding} are realized
by Hamiltonian symplectomorphisms.
\end{remark}

\bigskip
\textbf{Hypothesis (H3):}
\emph{Given a locally free $\mathbb{S}^1$--action $\{R_{\alpha}\}_{\alpha \in \mathbb{S}^1}$ on a manifold $N$, an open set $V \subset N$
and $\varepsilon > 0$, there exists an element $f \in \G(N)$ such that
$$\lambda(\{\alpha \in \mathbb{S}^1: R_{\alpha}(y) \notin f^{-1}(V)\}) < \varepsilon$$
for every $y \in N$ ($\lambda$ denotes the Lebesgue measure in $\mathbb{S}^1$ of total mass 1).}
\medskip

As was said before, there are some technicalities difficulties to overcome in the proof even in the case
$\G(N) = \diff(N)$ or $\diff_{\mu}(N)$ (see \cite{fathiherman}). Notice also that (H3) is out of reach
in the symplectic case. Indeed, the volume is always preserved under symplectic transformations so
it is not possible to modify $V$ to swallow ``most'' of all the orbits.

Henceforth, assume that $(N^{2n+1}, \xi)$ is a contact manifold, $R_{\alpha}$ is a contactomorphism
 for every $\alpha \in \mathbb{S}^1$ and $\G(N) = \cont_0(N, \xi)$.
Let $\{(U_j, \theta_j\colon U_j \to \theta_j(U_j) \subset \R^{2n+1})\}$ be a finite atlas of $N$ by Darboux charts,
that is, $\theta_j\colon (U_j, \xi_{|U_j}) \to (\R^{2n+1}, \xi_{0})$ is a contactomorphism.
The standard contact distribution in $\R^{2n+1}$ is denoted $\xi_0 = \ker(dz - \sum_{k=1}^n y_k dx_k)$.
The following packing result is well--known in the field of contact geometry,
a proof accessible to non--experts is presented in Section \ref{sec:packing}.
The term \emph{cuboid} is here used to name the closure of a domain in $\R^{2n+1}$ enclosed by
$2n+1$ couples of parallel hyperplanes in general position.

\begin{lemma}\label{lem:contactinflating}
Let $p \in \R^{2n+1}$, $V_p$ a neighborhood of $p$ and $Y$ be a vector field in $V_p$ such that $Y(p) \neq 0$.
There exists a cuboid $C$ centered at $p$ such that
\begin{itemize}
\item $C$ is contained in $V_p$,
\item $Y$ is transverse to the faces of $C$ and
\item for any neighborhood $W$ of $\partial C$ and any ball $B \subset C$ centered at $p$
it is possible to find a Hamiltonian contactomorphism $\varphi \colon (\R^{2n+1}, \xi_0) \to (\R^{2n+1}, \xi_0)$ such that:
\begin{enumerate}
\item $\supp(\varphi) \subset C$.
\item $\varphi(B)$ contains $C \setminus W$.
\end{enumerate}
\end{itemize}
\end{lemma}

For every $j$ and every $x \in U_j$
apply Lemma \ref{lem:contactinflating} to $\theta_j(x)$ and the vector field $(\theta_j)_*(X)$
to obtain a cuboid $Q^j_x$, centered at $x$, which is further assumed to lie within $\theta_j(U_j)$.
Define $C^j_x = \theta_j^{-1}(Q^j_x)$ and note that $X$ is transversal to $\partial C^j_x$.
Since $N$ is compact and $\{\interior(C^j_x): x \in U_j\}$ is an open cover of $N$,
there is a finite set of different points $\{p_i\}_{i = 1}^m$ whose associated $C_{p_i}$ cover $N$.
By transversality, any orbit of the action meets $\Delta = \bigcup_{i = 1}^m \partial C_{p_i}$ at a finite number of points.

\begin{lemma}\label{lem:skeleton}
If $W$ is a sufficiently small neighborhood of $\Delta$ then
$$\lambda(\{\alpha \in \mathbb{S}^1: R_{\alpha}(y) \in W\}) < \varepsilon$$
for every $y \in N$.
\end{lemma}
\begin{proof}
For any $y_0 \in N$, the orbit $\mathcal{O}(y_0) = \{R_{\alpha}(y_0): \alpha \in \mathbb{S}^1\}$ meets $\Delta$ in finitely many points,
$R_{\alpha_1}(y_0), \ldots, R_{\alpha_k}(y_0)$. Take a neighborhood $U$ of the union of these points
small enough so that $\lambda(\{\alpha \in \mathbb{S}^1: R_{\alpha}(y) \in U\}) < \varepsilon$ for every $y \in N$.
Clearly, if $W^{y_0}$ is a sufficiently small neighborhood of $\Delta$, $W^{y_0} \cap \mathcal{O}(y_0) \subset U$
and the same is true for orbits of points in a neighborhood of $y_0$. A compactness argument yields the result.
\end{proof}

Fix $W$ from the previous lemma and note that to conclude (H3) it is enough to inflate $V$ to cover $N \setminus W$.
This strategy splits into two steps. Firstly, choose $s$ different points $\{q_i\}$ in $V$. By Theorem \ref{thm:transitivity}
there is a Hamiltonian contactomorphism $\phi\colon (N, \xi) \to (N, \xi)$ such that $\phi(p_i) = q_i$.
Then, $\{p_i\} \subset \phi^{-1}(V)$.
Let $B_i$ be a small ball centered at $p_i$ and contained in $\phi^{-1}(V)$.

Recall from Lemma \ref{lem:contactinflating} the properties of $Q^i_x$, which are inherited by
$C_{p_i}$. In particular, there are Hamiltonian contactomorphisms $\varphi_i\colon (N, \xi) \to (N, \xi)$
supported in $C_{p_i}$ such that $\varphi_i(B_i) \supset C_{p_i} \setminus W$.
Consider
$\varphi = \varphi_1 \circ \ldots \circ \varphi_m.$

\begin{lemma}
$$\varphi \left( \cup_{i = 1}^s B_{i} \right) \cup W = N.$$
\end{lemma}
\begin{proof}
For a point $p$ in the support of $\varphi_i$ there are two non--exclusive possibilities:
$\varphi_i(p) \in W$ or $p \in B_{i}$. The conclusion follows from the fact that any point of $N$
belongs to some $C_{p_i}$ hence to the support of one or more $\varphi_i$.
\end{proof}

As a consequence, $\varphi \circ \phi^{-1} (V)$ contains $N \setminus W$. Thus, Lemma \ref{lem:skeleton} concludes (H3).

\begin{remark}
Notice that the map $\varphi \circ \psi$ is the composition of two Hamiltonian contactomorphisms.
\end{remark}

\section{Packing lemmas}\label{sec:packing}

This section discusses the two packing results (Proposition \ref{prop:folding} for (H2) in the symplectic case and
Lemma \ref{lem:contactinflating} for (H3) in the contact case) which eventually led to the proofs of the main theorems.

\subsection{Contact}

In the contact case, the goal is to construct a ``box'' within which any small ball may be inflated
(by a contact transformation) to take up all the space inside but a small margin. This is an easy consequence
of the basic fact that special dilations preserve the standard contact structure.

\begin{lemmacontactinflating}
Let $p \in \R^{2n+1}$, $V_p$ a neighborhood of $p$ and $Y$ be a vector field in $V_p$ such that $Y(p) \neq 0$.
There exists a cuboid $C$ centered at $p$ such that
\begin{itemize}
\item $C$ is contained in $V_p$,
\item $Y$ is transverse to the faces of $C$ and
\item for any neighborhood $W$ of $\partial C$ and any ball $B \subset C$ centered at $p$
it is possible to find a Hamiltonian contactomorphism $\varphi \colon (\R^{2n+1}, \xi_0) \to (\R^{2n+1}, \xi_0)$ such that:
\begin{enumerate}
\item $\supp(\varphi) \subset C$.
\item $\varphi(B)$ contains $C \setminus W$.
\end{enumerate}
\end{itemize}
\end{lemmacontactinflating}
\begin{proof}
If $p = (\x_0, \y_0, z_0)$, the affine map
\[
\tau(\x, \y, z) = (\x - \x_0, \y - \y_0, z - z_0 + \x_0 \cdot (\y - \y_0))
\]
is a contactomorphism in $(\R^{2n+1}, \xi_{0})$ which maps $p$ to the origin.
Thus, assume without lose of generality $p = \0$.

The vector field $V(\x, \y, z) = (\x, \y, 2z)$ is a contact vector field because its flow
$\psi_t(\x, \y, z) = (e^t\x, e^t\y, e^{2t}z)$ preserves the standard contact structure $\xi_{0}$.
Denote $H_V$ the contact Hamiltonian associated to $V$.
Let $C_r$ be the cuboid of size $r > 0$ centered at $\0$ and
generated by the set of linear 1--forms $\{\lambda_1 = dx_1, \ldots, \lambda_{2n+1} = dz\}$, that is,
\[
C_r = \{v \in \R^{2n+1} : |\lambda_i(v)| < r \enskip \forall i\}.
\]
A computation shows that $V$ points outwards $C_{r}$.
As a consequence, the image under the flow $\psi_t$ of any neighborhood $B$ of the origin eventually covers $C_{r}$.
Note that the statement remains valid as long as the vector field $V$ points outwards every such cuboid.
In case $\lambda_i(Y(\0)) = 0$, replace $\lambda_i$ by a sufficiently close linear 1--form $\tilde{\lambda}_i$.
For $r_0 > 0$ sufficiently small, $V$ still point outwards the modified cuboids
$\widetilde{C}_r = \{v \in \R^{2n+1} : |\widetilde{\lambda}_i(v)| < r \enskip \forall i\}$
and the faces of $\widetilde{C}_r$,
$\partial \widetilde{C}_r$, are transversal to $Y$ if $r \le r_0$. Take $C = \widetilde{C}_{r_0}$.

Fix now a ball $B$ centered at \0 and a neighborhood $W$ of $\partial C$. As was noticed before,
$\psi_t(B) \subset C \setminus W$ for large $t$.
Consider now a smooth function $H\colon \R^{2n+1} \to \R$ equal to $H_V$ inside $C \setminus W$ which vanishes outside $C$.
The contact vector field $V'$ associated to $H$ is then equal to $V$ inside $C \setminus W$ and vanishes outside $C$.
Consequently, the flow $\varphi_t$ generated by $V'$ satisfies the properties in the statement.
\end{proof}

\subsection{Symplectic}

This subsection contains the proof of Proposition \ref{prop:folding}, that is,
it is devoted to show how to squeeze a large codimension--1 disk into a small ball in a symplectic fashion.
Denote by $B^2(r)$ the closed 2--ball of radius $r$ and $P^{2n}(r_1, \ldots, r_n) = B^2(r_1) \times \ldots \times B^2(r_n)$.
The proof presented here adapts the following non--trivial result to answer the question.

\begin{lemma}\label{lem:folding}
Given $s, \rho > 0$ there exists $\eta = \eta(s, \rho) > 0$ and a Hamiltonian symplectomorphism $\phi$
such that $\phi$ embeds $B^2(\eta) \times B^2(s)$ into $B^2(\rho) \times B^2(1)$.
Furthermore, $\phi$ can be assumed to be supported in $B^2(c \rho) \times B^2(s + c)$ for a constant $c > 1$
independent of $s, \rho$.
\end{lemma}

There exist several approaches to this result in the literature. One could use the $h$--principle for isosymplectic embeddings to
obtain an embedding of the disk $B^2(s)$ and then extend it to a neighborhood thanks to the Symplectic Neighborhood Theorem.
The parametric version of this symplectic embedding theorem (see \cite[Section 12.1]{eliashberg}) would provide the result.
However, we prefer a more hands--on approach using symplectic folding.
Following Schlenk (\cite[Remark 3.3.1]{schlenk}), the nature of this deformation is local
and each step of the construction is induced by a Hamiltonian flow.
Indeed, a careful look through the folding shows that all deformations take place inside an arbitrary neighborhood
of the figures apart from the stretching in the base, where an extra space proportional to the size of the fibers is required.
The constant $c$ in the lemma is introduced to make up for it.

For our purpose, the extra space around the codimension--1 disk in which the transformation is supported must be
arbitrarily small. The following technical lemma shows how to
scale properly Lemma \ref{lem:folding} and to fold the disk in $n-1$ directions.

\begin{lemma}\label{lem:sabana}
Given $r, \varepsilon > 0$, for every $0 < \delta < \varepsilon$, there exists $\sigma > 0$ and a Hamiltonian
symplectomorphism $\varphi$ such that $\varphi$ embeds $P^{2n}(\sigma, r, \ldots, r)$
into $P^{2n}(\delta, \ldots, \delta)$ and the support of $\varphi$ is contained in
$P^{2n}(\varepsilon, r+\varepsilon, \ldots r+\varepsilon)$.
\end{lemma}
\begin{proof}
The goal is to find $\sigma$ so that the result of folding in each of the $n-1$ ``thick'' directions the thin polydisk
is contained into the target cube. The squeezing map of Lemma \ref{lem:folding} has to be scaled properly.

Set $\rho_1 = 1$ and consider $\lambda_1 > 0$ small enough so that the following inequalities are satisfied for
$\rho = \rho_1$ and $\lambda = \lambda_1$:
\begin{equation}\label{eq:1}
\lambda, \lambda \rho < \delta, \enskip \enskip \enskip \lambda c, \lambda c \rho < \varepsilon,
\end{equation}
where $c$ comes from Lemma \ref{lem:folding}.
Define $s_1 = r / \lambda_1$. Lemma \ref{lem:folding} yields $\rho_2 := \eta(s_1,\rho_1)$ and $\phi_{1}$.
The map $\widehat{\phi}_{1}(x) = \lambda_1 \phi_{1}(x/\lambda_1)$ defines a Hamiltonian symplectomorphism
such that:
\begin{itemize}
\item $\widehat{\phi}_{1}$ embeds $B^2(\lambda_1 \rho_2) \times B^2(\lambda_1 s_1)$ into
$B^2(\lambda_1 \rho_1) \times B^2(\lambda_1)$ so,
the set of inequalities (\ref{eq:1}) implies $B^2(\lambda_1 \rho_2) \times B^2(r)$ is sent inside $P^4(\delta, \delta)$.
\item The support of $\widehat{\phi}_{1}$ is contained in $B^2(\lambda_1 c \rho_1) \times B^2(\lambda_1(s_1 + c))$
which, by (\ref{eq:1}), is in turn contained in $B^2(\varepsilon) \times B^2(r + \varepsilon)$.
\end{itemize}
Setting $\sigma = \lambda_1 \rho_2$ would already prove the lemma for $n = 2$. Let us continue
the argument pursuing the general case.

For $k \ge 2$, define inductively $\rho_k = \eta(s_{k-1}, \rho_{k-1})$,
and $\lambda_k > 0$ smaller than $\lambda_{k-1}$ and such that inequalities (\ref{eq:1})
are satisfied for $\rho = \rho_k$ and $\lambda = \lambda_k$.
Define $\widehat{\phi}_{k} = \lambda_k \phi_{k}(x/\lambda_k)$ and $s_k = r/\lambda_k$.
For every $k \ge 1$, the following properties are satisfied:
\begin{itemize}
\item $\widehat{\phi}_{k}$ embeds $B^2(\lambda_k \rho_{k+1}) \times B^2(r) =
B^2(\lambda_k \eta(s_k, \rho_k)) \times B^2(\lambda_k s_k)$ inside $B^2(\lambda_k \rho_k) \times B^2(\lambda_k)$.
\item The support of $\widehat{\phi}_{k}$ is contained in $B^2(\lambda_k c) \times B^2(\lambda_k(s_k + c))$
so, by (\ref{eq:1})
\begin{equation}\label{eq:2}
\supp(\widehat{\phi}_{k}) \subset B^2(\varepsilon) \times B^2(r + \varepsilon).
\end{equation}
\end{itemize}
For any $1 \le k \le n-1$, $E_k$ denotes the linear subspace of $\R^{2n} = \R^2 \times \ldots \times \R^2$
spanned by the $1^{st}$ and $(n-k+1)^{th}$ factors.
Define $\widehat{\varphi}_{k}$ as the map which acts as $\widehat{\phi}_{k}$ in $E_k$ and as the identity in the other directions.
Evidently, $\widehat{\varphi}_{k}$ is again a Hamiltonian symplectomorphism.
In view of (\ref{eq:2}),
after a suitable cut-off we can obtain another Hamiltonian symplectomorphism $\varphi_k$ which coincides with $\widehat{\varphi}_k$
in $P^{2n}(\delta, r, \ldots, r)$ and is supported on $P^{2n}(\varepsilon, r + \varepsilon, \ldots, r + \varepsilon)$.
Define $\sigma = \lambda_{n-1} \rho_{n}$. Then,
\begin{equation*}
\left.\begin{aligned}
& P^{2n}(\sigma, r, r, \ldots, r) = P^{2n}(\lambda_{n-1}\rho_{n}, \lambda_{n-1}s_{n-1}, r, \ldots, r)
\xhookrightarrow{\varphi_{n-1}}\\
& \xhookrightarrow{\varphi_{n-1}} P^{2n}(\lambda_{n-1} \rho_{n-1}, \lambda_{n-1}, r, \ldots, r)\\
& P^{2n}(\lambda_{n-1} \rho_{n-1}, \lambda_{n-1}, r, \ldots, r)
\subset P^{2n}(\lambda_{n-2} \rho_{n-1}, \delta, r, \ldots, r) \xhookrightarrow{\varphi_{n-2}} \ldots \\
& \qquad \qquad \qquad \qquad \qquad \ldots \xhookrightarrow{\varphi_{2}} P^{2n}(\lambda_2 \rho_2, \delta, \ldots, \delta, r) \subset
P^{2n}(\lambda_1 \rho_2, \delta, \ldots, \delta, r) \\
& P^{2n}(\lambda_1 \rho_2, \delta, \ldots, r) \xhookrightarrow{\varphi_{1}} P^{2n}(\lambda_1 \rho_1, \delta, \ldots, \delta)
\subset P^{2n}(\delta, \ldots, \delta).
\end{aligned}\right.
\end{equation*}
Thus, in order to conclude the lemma it suffices to define
$$\varphi = \varphi_1 \circ \ldots \circ \varphi_{n-1}.$$
\end{proof}

\begin{propositionfolding}
Let $r > 0$, $D = \{0\} \times [-r, r] \times B^2(r) \times \ldots \times B^2(r) \subset \R^{2n}$
and $\varepsilon > 0$. For any $\delta > 0$, there exists a Hamiltonian symplectomorphism $\psi$
with support in
$[-\varepsilon, \varepsilon] \times [-r-\varepsilon, r + \varepsilon] \times B^2(r + \varepsilon) \times
\ldots B^2(r + \varepsilon)$ such that $\psi(D) \subset P^{2n}(\delta, \ldots, \delta)$.
\end{propositionfolding}
\begin{proof}
Fix $\sigma$ from the previous lemma.
The Hamiltonian $H(\x, \y) = -x_1 y_1$ induces a flow in $\R^{2n}$ which carries $\{0\} \times [-r, r] \times A^{2n-2}$ onto
$\{0\} \times [-\sigma, \sigma] \times A^{2n-2}$, for arbitrary $A^{2n-2}$.
Applying an appropriate cut-off to $H$ we can assume the flow is supported in
$[-\varepsilon, \varepsilon] \times [-r-\varepsilon, r + \varepsilon]  \times B^2(r+ \varepsilon) \times \ldots \times B^2(r+ \varepsilon)$.
It is enough to compose the time--$t$ map of the flow, for large $t > 0$,
with $\varphi$ from Lemma \ref{lem:sabana} to obtain the desired map.
\end{proof}

\end{document}